\documentclass[[amscd,amssymb,verbatim,12pt]{amsart}
\usepackage{amssymb, latexsym, amsmath, amscd, array, graphicx}
\usepackage{hyperref}
\usepackage[all]{xy}

\swapnumbers \numberwithin{equation}{section}

\theoremstyle{plain}

\newtheorem{thm}{Theorem}[section]

\newtheorem{lemma}[thm]{Lemma}

\newtheorem{prop}[thm]{Proposition}
\newtheorem{cor}[thm]{Corollary}

\theoremstyle{definition}
\newtheorem{defn}[thm]{Definition}

\newtheorem{remark}[thm]{Remark}

\newtheorem{question}[thm]{Question}

 \newcommand{\Wi}{\widetilde}

\DeclareMathOperator{\Ker}{{\rm Ker}}

\def\Z{{\mathbb Z}}
\def\Q{{\mathbb Q}}

\def\N{{\mathbb N}}

\def\1{\hbox{\rm\rlap {1}\hskip.03in{\rom I}}}
\def\Bbbone{{\rm1\mathchoice{\kern-0.25em}{\kern-0.25em}
{\kern-0.2em}{\kern-0.2em}I}}

\long\def\forget#1\forgotten{} %

\newcommand\ver[1]{\marginpar{\tiny Changed in Ver \VER}}

\newcommand{\mc}{ \text {mc}}

\date{\today}

\begin{document}

\title[On Gromov's conjecture]{On Gromov's conjecture for right-angled Artin groups}

\author[A.~Dranishnikov]{Alexander  Dranishnikov$^{1}$} 

\author[S.~Howladar]{Satyanath Howladar}

\thanks{$^{1}$Supported by Simons Foundation}

\address{Alexander N. Dranishnikov, Department of Mathematics, University
of Florida, 358 Little Hall, Gainesville, FL 32611-8105, USA}
\email{dranish@ufl.edu}

\address{Satyanath Howladar, Department of Mathematics, University
of Florida, 358 Little Hall, Gainesville, FL 32611-8105, USA}
\email{showladar@ufl.edu}



\begin{abstract}

We explore in terms of KO-homology the 2nd obstruction for the universal covering $\Wi M$ of a closed spin $n$-manifold $M$ with positive scalar curvature (PSC) to have macroscopic dimension $\le n-2$. As a corollary we prove that the universal covering $\Wi M$ of an $n$-dimensional closed spin PSC manifold $M$ with $\pi_1(M)\in$ RAAG has macroscopic dimension $\dim_{mc}\Wi M\le n-2$. This confirms Gromov's conjecture  in the case of  RAAG (right-angled Artin group) fundamental groups.
\end{abstract}



\maketitle
\section {Introduction}

The notion of macroscopic dimension was introduced by M.
Gromov~\cite{G2} to study topology of manifolds with a positive scalar curvature  metric.

\begin{defn}
 A metric space $X$ has the macroscopic dimension $\dim_{\mc} X \leq k$ if
there is a uniformly cobounded proper map $f:X\to K$ to a $k$-dimensional simplicial complex.
Then $\dim_{mc}X=m$ where $m$ is minimal among $k$ with $\dim_{\mc} X \leq k$.
\end{defn}
\smallskip
A map of a metric space $f:X\to Y$ is uniformly cobounded if there is a uniform upper bound on the diameter of preimages $f^{-1}(y)$, $y\in Y$.

\

{\bf Gromov's Conjecture.} {\it The macroscopic dimension of the universal covering $\Wi M$ of  a closed  positive scalar curvature  $n$-manifold $M$ satisfies the inequality $\dim_{mc}\Wi M\leq n-2$ for the metric on $\Wi M$ lifted from $M$}.

\

The main examples supporting Gromov's Conjecture are $n$-manifolds of the form $M=N\times S^2$. They admit metrics with PSC in view of
the formula $Sc_{x_1,x_2}=Sc_{x_1}+Sc_{x_2}$ for the cartesian product $(X_1\times X_2,\mathcal G_1\oplus\mathcal G_2)$
of two Riemannian manifolds $(X_1,\mathcal G_1)$ and $(X_2,\mathcal G_2)$ and the fact that while $Sc_N$ is bounded $Sc_{S^2}$ can be chosen
to be arbitrary large. Clearly, the projection $p:\Wi M=\Wi N\times S^2\to\Wi N$ is a proper uniformly cobounded map to a $(n-2)$-dimensional manifold
which can be triangulated.
Hence, $\dim_{mc}\Wi M\le n-2$.

Since $\dim_{mc}X=0$ for every compact metric space, the
Gromov Conjecture holds trivially for  manifolds with finite fundamental groups. Therefore, Gromov's conjecture  is about manifolds with infinite fundamental groups.
We note that even a weaker version of the Gromov Conjecture, that is $\dim_{mc}\Wi M\leq n-1$, for positive scalar curvature manifolds looks out of reach. This is because it implies the Gromov-Lawson's Conjecture: {\em A closed aspherical manifold cannot carry a metric of positive scalar curvature}.
The latter is known as a twin sister of the famous Novikov Higher Signature conjecture. The real analytic version of the Novikov conjecture is called the Strong Novikov 
Conjecture. It is a conjecture for discrete groups $\pi$: {\em The analytic assembly map $\alpha: KO_*(B\pi) \to KO_*(C^*(\pi))$ is a monomorphism, where $C^*(\pi)$ is reduced $C^*$-algebra of $\pi$.}
We note that the Strong Novikov conjecture for a group $\pi$ implies the Higher Signature conjecture for all manifolds with the fundamental group $\pi$.

The (Strong) Novikov's conjecture is  proven for large classes of groups (see the survey~\cite{Yu}). 
The Gromov Conjecture is proven  in far fewer cases 
\cite{Bo1},\cite{Bo3}, \cite{Dr1},\cite{Dr2},
\cite{BD1},\cite{BD2},\cite{DD}. In view of the above connection between the Gromov and the Novikov conjectures, it makes sense to state the Gromov Conjecture 
for groups instead of manifolds as the Gromov Conjecture for all positive scalar curvature manifolds with a given fundamental group $\pi$. Additionally, to make Gromov's conjecture more accessible we consider it in this paper for spin manifolds  with the fundamental group $\pi$ satisfying the Strong Novikov Conjecture. 
In this paper we proved the inequality  $\dim_{mc}\Wi M\leq n-1$ for spin manifolds with the fundamental group $\Gamma$ that satisfies the Strong Novikov Conjecture and has its integral homology $H_*(\Gamma)$ torsion free (Corollary~\ref{1 step}).

Gromov defined {\em inessential manifolds} $M$ as those for which a classifying map $u:M\to B\Gamma$ of the universal covering $\Wi M$ can be deformed to the $(n-1)$-skeleton $B\Gamma^{(n-1)}$ where $n=\dim M$. Clearly, for an inessential $n$-manifold $M$  we have $\dim_{mc}\Wi M\leq n-1$.

We call an $n$-manifold {\em strongly inessential} if a classifying map of its universal covering $u:M\to B\Gamma$ can be deformed to the $(n-2)$-skeleton.
Since for strongly inessential $n$-manifolds $\dim_{mc}\Wi M\le n-2$, the following conjecture implies the Gromov's conjecture.

{\bf Strong Gromov's Conjecture.} {\it  A closed  positive scalar curvature  manifold $M$ is strongly inessential.}

We note that for a finite cover $M'\to M$ of a PSC manifold $M$ the manifold $M'$ has PSC and the universal cover of $M'$ coincides with the universal cover of $M$, 
Therefore, the Strong Gromov Conjecture for a group $\Gamma'$  implies the original Gromov Conjecture
for all groups $\Gamma$ containing $\Gamma'$ as a finite index subgroup.

Dmitry Bolotov showed~\cite{Bo2} that there are inessential manifolds which are not strongly inessential.

\

In this paper we prove the Strong Gromov Conjecture for spin $n$-manifolds with $n>4$, for some classes of groups, in particular, for right-angled Artin groups (RAAGs).

\section {Preliminaries}

We recall that for
every spectrum $E$ there is a connective cover $e \to E$, that is  a spectrum $e$ with the
morphism $ e\to  E$ that induces isomorphisms of homotopy  groups $\pi_i(e) \to \pi_i(E)$ for $i\ge 0$ and
with $\pi_i(e) = 0$ for $i < 0$ . By $KO$ we denote the spectrum for real K-theory, by
$ko$ its connective cover, and by $per : ko\to KO$ the corresponding 
morphism of spectra. For more details we refer to ~\cite{Ru}. We use the standard notation $\pi^s_*$ for the stable homotopy groups.
We use
both notations for an $E$-homology of a space $X$: old-fashioned
$E_*(X)$ and modern $ H_*(X;E)$. We recall that $KO_n(pt)=\Z$ if
$n=0$ or $n=4$ mod 8, $KO_n(pt)=\Z_2$ if $n=1$ or $n=2$ mod 8, and
$KO_n(pt)=0$ for all other values of $n$. By $\mathbb S$ we denote the
spherical spectrum. Note that for any ring spectrum $E$ there is a
natural morphism $\mathbb S\to E$ which leads to the natural transformation
of the stable homotopy to $E$-homology $\pi^s_*(X)\to H_*(X;E)$.

Let $M(\Z_m,n)=S^n\cup_\phi D^{n+1}$ denote the Moore space where $\phi:S^n\to S^n$ is a map of degree $m$.
\begin{prop}\label{moor}
The homomorphism $$per: ko_i(M(\Z_m,n))\to KO_i(M(\Z_m,n))$$ is an isomorphism for $i\ge n$.
\end{prop}
\begin{proof}
We apply the Five Lemma to the diagram with $M=M(\Z_m,n)$
$$
\begin{CD}
ko_i(S^n)@>m>> ko_i(S^n) @>>> ko_i(M) @>>> ko_{i-1}(S^n) @>m>> ko_{i-1}(S^n)\\
@V\cong VV @V\cong VV @VVV @V\cong VV @V\cong VV\\
KO_i(S^n)@>m>> KO_i(S^n) @>>> KO_i(M) @>>> KO_{i-1}(S^n) @>m>> KO_{i-1}(S^n)\\
\end{CD}
$$
to get the result. Note that for $i=n$ on the right there will be isomorphisms, since $KO_{-1}(pt)=0$.
\end{proof}

The following proposition is taken from~\cite{BD1}.

\begin{prop}\label{pi-ko-iso}
	The natural transformation $\pi_*^s(pt)\to ko_*(pt)$ induces an
	isomorphism $\pi_{n}^s(K/K^{(n-2)})\to ko_{n}(K/K^{(n-2)})$ for any CW
	complex $K$.
\end{prop}

We use the following  well-known fact (see 4C, \cite{Ha}).
\begin{prop}\label{moore} Let $X$ be an $(n-1)$-connected $(n+1)$-dimensional CW complex.
Then $X$ is homotopy equivalent to the wedge of spheres of dimensions $n$ and $n+1$
together with the Moore spaces $M(\Z_m, n)$.
\end{prop}

\

Next, we recall the following classical result (see~Corollary 6.10.3, \cite{TtD}):
\begin{thm}\label{h-excision}
	Suppose that a CW-complex pair $(X,A)$ satisfies the conditions $\pi_i(X)=0$ for $i\le m$ and $\pi_i(A)=0$ for $i\le m-1$ with $m\ge 2$. 
Then the quotient map $q:(X,A)\to (X/A,\ast)$
	induces isomorphisms $q_*:\pi_i(X,A)\to\pi_i(X/A)$ for $i\le 2m-1$.
\end{thm}

\subsection{Inessential manifolds}

Note that for an inessential $n$-manifold $M$ we have $\dim_{mc}\Wi M\le n-1$. Indeed, a lift $\Wi{u_M}:\Wi M\to E\Gamma^{(n-1)}$ of a classifying
map is a uniformly cobounded proper map to an $(n-1)$-complex.

The first step in any  proof of the Strong Gromov Conjecture would be establishing the  inessentiality of a positive scalar curvature  manifold.  
We recall that the inessentiality of a manifold can be characterized as follows~\cite{Ba} (see also~\cite{BD1}, Proposition 3.2). 
\begin{thm}\label{ref} Let $M$ be a closed oriented $n$-manifold. 
Then the following are equivalent:

1. $M$ is inessential;

2. $(u_M)_*([M])=0$ in $H_n(B\Gamma)$ where $[M]$ is the fundamental class of $M$ and $\Gamma=\pi_1(M)$.  
\end{thm}

Gromov calls an orientable closed $n$-manifold $M$ {\em rationally inessential} if  $(u_M)_*([M])=0$ in $H_n(B\Gamma;\Q)$.

It is known that for a manifold to be spin is equivalent to be orientable  for any of  the K-theories: complex $KU$, real $KO$, or  their connective covers $ku$ or $ko$ ~\cite{Ru}.
Suppose that a closed $n$-manifold $M$ is orientable with respect to a generalized homology theory $E$. We call the manifold $M$ to be {\em  $E$-inessential} if
it is $E$-orientable and  $u_*([M]_E)=0$ where $[M]_E\in H_n(M;E)$
is the $E$-fundamental class. Note that this property does not depend on the choice of  $E$-orientation.

We recall the Rosenberg index theorem~\cite{R1}:
\begin{thm}\label{Ros}
Suppose that the fundamental group $\Gamma$ of a positive scalar curvature spin manifold $M$ satisfies the Strong Novikov Conjecture. Then $M$ is $KO$-inessential.
\end{thm}
Note that in view of the Chern character isomorphism a $KO$-inessential manifold is rationally inessential.
As Marcinkowski's example shows~\cite{M}, the rational inessentiality of an $n$-manifold  does not imply the inequality $\dim_{mc}\Wi M \le n-1$. This makes difficult
the first step in the Gromov's conjecture even in the case of groups satisfying the Strong Novikov Conjecture. 

\begin{prop}[~\cite{BD1}, Proposition 3.3]\label{ko-in}
A $ko$-inessential manifold is inessential. 
\end{prop}

\begin{question}\label{Q0}
When a $KO$-inessential manifold is inessential ?
\end{question}
In view of Proposition~\ref{ko-in} this is true for manifolds with the fundamental group $\Gamma$ for which the map $per:ko_n(B\Gamma)\to KO_n(B\Gamma)$ is injective.\\

In view of the Rosenberg Index theorem, we obtain the following
\begin{prop}
Suppose that the fundamental group $\Gamma$ of a positive scalar curvature closed spin $n$-manifold $M$ has the following properties:

(1) It satisfies the Strong Novikov Conjecture;

(2) The  map $per:ko_n(B\Gamma)\to KO_n(B\Gamma)$ is injective.

Then $M$ is inessential.
\end{prop}
The conditions (1) and (2 for all $n$) are known as the Rosenberg-Stolz conditions. For closed spin manifolds with the fundamental group satisfying Rosenberg-Stoltz conditions
there is a vanishing index criterion when they could carry a metric with positive scalar curvature~\cite{R2}. In other words the Gromov-Lawson-Rosenberg conjecture holds true for such groups.\\

\begin{prop}\label{torsion free}
Suppose that a group $\Gamma$ has torsion free integral homology. Then every $KO$-inessential manifold with the fundamental group $\Gamma$ is inessential.
\end{prop}
\begin{proof}
Since the $KO$-fundamental class $[M^n]_{KO}$ maps under the natural transformation $KO\to KU$ to a $KU$-fundamental class, every $KO$-inessential manifold is $KU$-inessential. 
In view of the diagram of the homology Chern character isomorphisms generated by a degree one map $f:M^n\to S^n$
$$
\begin{CD}
K_n(M^n)\otimes\Q @>ch>\cong> \bigoplus_{i\in\Z} H_{n+2i}(M^n)\otimes\Q\\
@Vf_*\otimes 1_{\Q}VV @ Vf_*\otimes 1_{\Q}VV\\
K_n(S^n)\otimes 1_{\Q} @>ch>\cong> \bigoplus_{i \in\Z}H_{n+2i}(S^n)\otimes\Q\\
\end{CD}
$$
the image of a $KU$-fundamental class of $M^n$ under the Chern character has nontrivial $H_n(M^n)$ component.
Then the bottom arrow in the diagram
$$
\begin{CD}
H_n(M^n) @>u_*>> H_n(B\Gamma)\\
@VVV @VVV\\
H_n(M^n)\otimes \Q @>u_*\otimes 1_{\Q}>> H_n(B\Gamma)\otimes\Q\\
\end{CD}
$$
is  a zero homomorphism. The right vertical homomorphism is injective, since $H_n(B\Gamma)$ is torsion free.
Therefore, the top arrow is zero.
\end{proof}
\begin{cor}\label{1 step}
Suppose that the fundamental group $\Gamma$ of a positive scalar curvature spin manifold $M$ satisfies the Strong Novikov conjecture and all homology groups $H_*(\Gamma)$
are torsion free. Then $$\dim_{mc}\Wi M\le n-1.$$
\end{cor}
\begin{proof}
We apply Rosenberg's theorem (Theorem~\ref{Ros}) and Proposition~\ref{torsion free}.
\end{proof}

In~\cite{BD1} we proved the following addendum to Theorem~\ref{ref}.
\begin{prop}[\cite{BD1}, Lemma 3.5]\label{ref2}
For an inessential manifold $M$  
with a CW complex structure a classifying map $u:M\to B\Gamma$ can be chosen such that 
$$u(M,M^{(n-1)})\subset (B\Gamma^{(n-1)},B\Gamma^{(n-2)}).$$
\end{prop}

\section{ The second obstruction}

Suppose that $\Gamma$ acts on an abelian group $\pi$. By $\pi_\Gamma=\pi\otimes_{\Z\Gamma}\Z$ we denote the group of coinvariants.
Since every $ko$-inessential manifold is inessential ( Proposition~\ref{ko-in}) and for every inessential $n$-manifold there is a classifying map $u:M\to B\Gamma$ with
 $u(M,M^{(n-1)})\subset (B\Gamma^{(n-1)},B\Gamma^{(n-2)})$ ( Proposition~\ref{ref2}), for $ko$-inessential manifolds we always will consider such 
classifying maps $u$.

The following lemma shows that the second step in the Strong Gromov's conjecture can be also reduced to a question about $ko$-inessentiality.

\begin{lemma}\label{2-obstruction}
Suppose that a closed spin $n$-manifold $M$, $n>4$, with $\Gamma=\pi_1(M)$  is $ko$-inessential in a way that $u_*([M]_{ko})=0$ in $ko_n(B\Gamma^{(n)})$ for a cellular classifying map
$u: M\to B\Gamma$. Then $M$ is strongly inessential.

Moreover, $M$ is strongly inessential if $j_*u_*([M]_{ko})=0$ where $j:B\Gamma^{(n)}\to B\Gamma^{(n)}/B\Gamma^{(n-2)}$ is the quotient map.
\end{lemma}
\begin{proof}
We may assume that $M$ has a CW complex structure with one $n$-dimensional cell. Let $\psi:D^n\to M$ be its characteristic map.

We show that the lifting problem

\[
  \xymatrix{
  M^{(n-1)} \ar[r] \ar[d]^\subset & B\Gamma^{(n-2)} \ar[d]^\subset \\
   M \ar[r]^{iu}\ar[ur] &  B\Gamma^{(n)}.  }
\]
has a solution. Here $i:B\Gamma^{(n-1)}\to B\Gamma^{(n)}$ is the inclusion.  It would  mean that there is a homotopy lift $\hat u: M\to B\Gamma^{(n-2)}$ of $i\circ u$ which agrees with $u$ on $M^{(n-2)}$.
Since $n\ge 4$, the map $\hat u$ induces an isomorphism of the fundamental groups and, hence, is  classifying map. 

Note that  the  $(n-1)$-homotopy groups of the homotopy  fiber $F$ of the inclusion $B\Gamma^{(n-2)}\to B\Gamma^{(n)}$ equals the relative $n$-homotopy 
group, $\pi_{n-1}(F)=\pi_n(B\Gamma^{(n)},B\Gamma^{(n-2)})$.
Then the first and the only obstruction to deform $u$ to $B\Gamma^{(n-2)}$ is defined by the cocycle
$c_u:C_n(M)\to\pi_{n-1}(F)$ represented by the composition 
$$\pi_n(D^n,\partial D^n)\stackrel{\psi_*}\to\pi_n(M,M^{(n-1)})\stackrel{u_*}\to\pi_n(B\Gamma^{(n-1)},B\Gamma^{(n-2)})\stackrel{i_*}\to\pi_n(B\Gamma^{(n)},B\Gamma^{(n-2)})$$ 
with the cohomology class
$o_u=[c_u]\in  H^n(M;\pi_n(B\Gamma^{(n)},B\Gamma^{(n-2)}))$.  By the Poincare Duality with local coefficients, the cohomology class $o_u$ is dual
to the homology class $$PD(o_u)\in H_0(M;\pi_n(B\Gamma^{(n)},B\Gamma^{(n-2)}))=\pi_n(B\Gamma^{(n)},B\Gamma^{(n-2)})_{\Gamma}$$
represented by $q_*i_*u_*\psi_*(1)$ where $$q_*:\pi_n(B\Gamma^{(n)},B\Gamma^{(n-2)})\to\pi_n(B\Gamma^{(n)},B\Gamma^{(n-2)})_{\Gamma}$$ 
is the projection onto the group of coinvariants.  

Note that $\pi_n(B\Gamma^{(n-1)},B\Gamma^{(n-2)})=\pi_n(E\Gamma^{(n-1)},E\Gamma^{(n-2)})$.
Below we will identify these groups. Denote by$$\tilde i:(E\Gamma^{(n-1)},E\Gamma^{(n-2)})  \to (E\Gamma^{(n)},E\Gamma^{(n-2)})$$ the inclusion induced by $i$.

Since $n\le 2(n-2)-1$, by 
Theorem~\ref{h-excision}, $$\pi_n(E\Gamma^{(n-1)},E\Gamma^{(n-2)})=\pi_n(E\Gamma^{(n-1)}/E\Gamma^{(n-2)}).$$
It is easy to see that $$\pi_n(E\Gamma^{(n-1)}/E\Gamma^{(n-2)})_\Gamma=\pi_n(B\Gamma^{(n-1)}/B\Gamma^{(n-2)}).$$
Similarly, $$\pi_n(E\Gamma^{(n)}/E\Gamma^{(n-1)})_\Gamma=\pi_n(B\Gamma^{(n)}/B\Gamma^{(n-1)}).$$
The homotopy exact sequence of the triple $(E\Gamma^{(n)},E\Gamma^{(n-1)},E\Gamma^{(n-2)})$ produces the following commutative diagram
$$
\begin{CD}
\pi_{n+1}(E\Gamma^{(n)},E\Gamma^{(n-1)})@>>>\pi_{n}(E\Gamma^{(n-1)},E\Gamma^{(n-2)}) @>\tilde i_*>> im(\tilde i_*) @>>> 0\\
@VVV @VVV @Vq_*VV @.\\
\pi_{n+1}(E\Gamma^{(n)},E\Gamma^{(n-1)})_\Gamma@>>>\pi_{n}(E\Gamma^{(n-1)},E\Gamma^{(n-2)})_\Gamma @>\tilde i_*\otimes_\Gamma 1_\Z>> im(\tilde i_*)_\Gamma @>>> 0\\
@V\cong VV @V\cong VV @V\xi VV @.\\
\pi_{n+1}(B\Gamma^{(n)}/B\Gamma^{(n-1)})@>>>\pi_{n}(B\Gamma^{(n-1)}/B\Gamma^{(n-2)}) @> i_*>> im( i_*) @>>> 0\\
\end{CD}
$$
where the middle row  is exact being obtained by tensor product by $\Z$ over $\Z\Gamma$ of an exact row.
By the Five Lemma the homomorphism $\xi$ is an isomorphism.

 Denote by $\bar u:M/M^{(n-1)}=S^n\to B\Gamma^{(n-1)}/B\Gamma^{(n-2)}$ the induced map. The following commutative diagram
$$
\begin{CD}
\pi_n(M,M^{(n-1)}) @>\tilde i_*u_*>>im(\tilde i_*) @>q_*>> im(\tilde i_*)_\Gamma @.\\
@A\psi_*AA @. @VV\xi V @.\\
\pi_n(D^n,\partial D^n) @>=>>\pi_n(M/M^{(n-1)}) @>i_*\bar u_*>> im(i_*)@>\subset>> \pi_n(B\Gamma^{(n)}/B\Gamma^{(n-2)})\\
\end{CD}
$$
implies that
$i_*\bar u_*(1)=\xi q_*\tilde i_*u_*\psi_*(1).$  Thus, $i_*\bar u_*(1)=0$ if and only if the obstruction $o_u$ vanishes. 

We show that $i_*\bar u_*(1)=0$.
The restriction $n>4$ imply that $\bar
u_*(1)$ survives to the $ko$-homology group:
$$
\begin{CD}
\pi_n(B\Gamma^{(n)}/B\Gamma^{(n-2)}) @>\cong >> \pi_n^s(B\Gamma^{(n)}/B\Gamma^{(n-2)}) @>\cong >>ko_n(B\Gamma^{(n)}/B\Gamma^{(n-2)}).\\
\end{CD}
$$

The first isomorphism is obtained by taking $n$-times suspension and using Freudenthal Theorem, and the second isomorphism by Proposition~\ref{pi-ko-iso}.
Then the  commutative diagram 
$$
\begin{CD}
\pi_n(S^n)  @>\cong>> ko_n(S^n)\\
@Vi_*\bar u_*VV @Vi_*\bar u_*VV \\
\pi_n(B\Gamma^{(n)}/B\Gamma^{(n-2)}) @>\cong>> ko_n(B\Gamma^{(n)}/B\Gamma^{(n-2)})
\end{CD}
$$
implies that $i_*\bar u_*(1)= 0$ for $ko_n$ if and only if $i_*\bar u_*(1)=0$ for $\pi_n$.

From the assumption and the diagram defined by the quotient maps $j':M\to M/M^{(n-1)}=S^n$ and $j:B\Gamma^{(n)}\to B\Gamma^{(n)}/B\Gamma^{(n-2)}$
\[
 \begin{CD}
ko_n(M)  @>u_*>>ko_n(B\Gamma^{(n)})\\
@Vj_*'VV @Vj_*VV\\
 ko_n(S^n)   @>i_*\bar u_*>> ko_n(B\Gamma^{(n)}/B\Gamma^{(n-2)}) \\
\end{CD}
\]
it follows that $i_*\bar u_*(1)=i_*\bar u_*j'_*([M]_{ko})=j_*u_*([M]_{ko})=0$.
\end{proof}

\begin{remark}
A slightly weaker lemma was proven  in~\cite{Dr1}. Namely,  a strong inessentiality of $M$ was proven under assumption of $j_*u_*([M]_{ko})=0$ in $ko_n(B\Gamma^{(n-1)}/B\Gamma^{(n-2)})$ instead of $ko_n(B\Gamma^{(n)}/B\Gamma^{(n-2)})$. Since $ko_*$ is a connective homology theory, for any $ko$-inessential $n$-manifold $M$ with the fundamental group $\Gamma$ we have $u_*([M]_{ko})=0$ in $ko_n(B\Gamma^{(n+1)}/B\Gamma^{(n-2)})$. Thus, in the current version of lemma on the second obstruction 
we are one step closer to an affirmative answer to the following question.
\end{remark}
\begin{question}\label{Q1}
Is every  $ko$-inessential manifold  strongly  inessential ?
\end{question}
We note that Bolotov's example~\cite{Bo2} is $ko$-essential.

\section{K-theory Stabilization Condition}

The following condition on K-theory of a group $\Gamma$ was introduced in~\cite{Dr1} In the proof of  the Strong Gromov Conjecture for abelian fundamental groups in~\cite{Dr1}
(see also~\cite{Bo1})
we considered the following property of groups $\Gamma$ denoted by (*) :
The inclusion homomorphism
$
  KO_*(B\Gamma^{(m)})\to KO_*(B\Gamma)
$
is injective for all $m$.

This property depends on choice of $B\Gamma$.
Here we introduce a weaker condition which is independent of the choice of $B\Gamma$ and sufficient for our purposes. 

First, we define this condition for spaces.
We state it for arbitrary generalized homology theory $h_*$.

{\bf 1-Step Stabilization property} for CW complex $X$: {\em For every $m$ the image of the inclusion homomorphism $h_*(X^{(m)})\to h_*(X)$  stabilizes at $X^{(m+1)}$.}

This property can be written as the formula $$\Ker(\phi_{m,m+1})_*=\Ker(\phi_{m,\infty})_*$$
where $\phi_{m,m+\ell}:X^{(m)}\to X^{(m+\ell)}$, $\ell=1,\dots,\infty$, denotes the inclusions. 

We note that every CW complex has the  1-Step Stabilization property for the ordinary homology $H_*$.

By default the 1-Step Stabilization property is assumed for $h_*$ in all dimensions, but
often we use this property for fixed $k$ or $k$ depending on $m$. We denote the set of such $k$ by $K(m)\subset\N$.
In such situation we say that $X$ has the 1-Step Stabilization property for $h_*$  restricted to $K(m)$.

For example, for the stable homotopy groups $\pi^s_*$ and, for that matter for any connective homology theory, the 1-Step Stabilization property
restricted to $K(m)=\{k\in\N\mid k\le m\}$ holds for
any CW complex.  In this paper we use the 1-Step Stabilization property for the real $K$-theory, $KO_*$, with $k=m,m+1$. Since the complex $K$-theory, $KU_*$ is 2-periodic, it is not a restriction on $k$, but in the case of $KO_*$ it is a potential restriction.

\begin{prop}\label{idependent}
The (restricted) 1-Step Stabilization property is a homotopy invariant. 
\end{prop}
\begin{proof}
Suppose that a CW complex $X_1$ has the 1-Step Stabilization property for $h_*$.
Let $f:X_1\to X_2$ be a homotopy equivalence with a homotopy inverse map $g$. We may assume that both maps $f$ and $g$ are cellular. 
Then  in view of the cellular approximation theorem the composition 
$$(f|_{X_1^{(m+1)}})\circ \phi_{m,m+1}\circ g|_{X_2^{(m)}}:X_2^{(m)}\to X_2^{(m+1)}$$ is homotopic to the inclusion $\phi'_{m,m+1}:X_2^{(m)}\to X_2^{(m+1)}$. 
Thus, we have  the following commutative diagram
$$
\begin{CD}
h_k(X_1^{(m)}) @>(\phi_{m,m+1})_*>> h_k(X_1^{(m+1)}) @>(\phi_{m+1,\infty})_*>> h_k(X_1)\\
@Ag|_*AA @Vf|_*VV @Ag_*A\cong A\\
h_k(X_2^{(m)}) @>(\phi'_{m,m+1})_*>> h_k(X_2^{(m+1)}) @>(\phi'_{m+1,\infty})_*>> h_k(X_2).\\
\end{CD}
$$
Assume that $(\phi'_{m,m+1})_*(a)\ne 0$. Then $(f|)_*(\phi_{m,m+1})_* (g|)_*(a)\ne 0$ and, hence, $(\phi_{m,m+1})_*(g|)_*(a)\ne 0$.
By the 1-step Stabilization property we obtain
$$
0\ne (\phi_{m+1,\infty})_*(\phi_{m,m+1})_* g|_*(a)=g_*(\phi'_{m+1,\infty})_* (f|)_*(\phi_{m,m+1})_*( g|)_*(a)$$
$$=g_*(\phi'_{m+1,\infty})_* (\phi'_{m,m+1})_* (a).
$$
Therefore, $(\phi'_{m+1,\infty})_* (\phi'_{m,m+1})_* (a)\ne 0$.
\end{proof}

\begin{prop}\label{quotients}
 If $X$ has the (restricted) 1-Step Stabilization property for $h_*$, then so does the quotient complex $Y=X/X^{(n)}$.
\end{prop}
\begin{proof}
We have to check the 1-Step Stabilization property for $Y$ only for $m\ge n+1$.
Let $a\in \Ker\{h_k(Y^{(m)})\to h_k(Y)\}$ and let $\bar a\in h_k(Y^{(m+1)})$ be the image of $a$.
We show that $\bar a=0$.
We note that $Y^{(m)}=X^{(m)}/X^{(n)}$.
From the diagram generated by the exact sequence of pairs $(X^{(m)},X^{(n)})$ and $(X,X^{(n)})$ it follows that $a\in\Ker\{\partial: h_k(Y^{(m)})\to h_{k-1}(X^{(n)})\}$. Hence, by exactness $a$ is the image of some $b\in h_k(X^{(m)})$. If  the image of $b$ in $h_k(X)$ is trivial, then by the 1-Step Stabilization property its image in $h_k(X^{(m+1)})$ is zero and, therefore, as one can see from the diagram below, $\bar a=0$,
$$
\begin{CD}
h_k(X^{(n)}) @>>> h_k(X^{(n)}) @>>> h_k(X^{(n)}) \\
@Vi_*VV @Vi_*VV @Vi_*VV\\
h_k(X^{(m)}) @>>> h_k(X^{(m+1)}) @>>> h_k(X)\\
@Vj_*VV @Vj_*VV @Vj_*VV\\
h_k(Y^{(m)}) @>>> h_k(Y^{(m+1)}) @>>> h_k(Y).\\
\end{CD}
$$
Suppose that the image $b'\in h_k(X)$ of $b$ is not zero. By exactness, there is $c\in h_k(X^{(n)})$ that maps by $i_*$ onto $b'$. Then the image of $b-i_*(c)$ in $h_k(X)$ is zero. By the 1-Step Stabilization, $\bar b-i_*(c)=0$ where $\bar b\in h_k(X^{(m+1)})$ is the image of $b$. Note that
$\bar a=j_*(\bar b)=j_*(\bar b-i_*(c))=0$.
\end{proof}

Since all classifying CW complexes $B\Gamma$ for a group $\Gamma$ are homotopy equivalent, we can define 1-Step Stabilization property
for groups. Thus, a group $\Gamma$ has  the {\bf 1-step Stabilization property} for a generalized homology theory $h_*$ if a classifying CW complex $B\Gamma$ has it.

We prove the following two propositions

\begin{prop}\label{KO-iness}
Suppose that a group $\Gamma$ has the 1-Step Stabilization property for $KO_*$ for all $m$, with $k=m$. Then every
closed  $KO$-inessential manifold $M$ with the fundamental group $\pi_1(M)=\Gamma$ is inessential.
\end{prop}

\begin{prop}\label{strong iness}
Suppose that an inessential and $KO$-inessential $n$-manifold $M$ has the fundamental group $\Gamma$, which satisfies the 1-Step Stabilization property for $KO_*$
for all $m$ with $k=m+1$. Then $M$ is strongly inessential.
\end{prop}

Proposition~\ref{KO-iness} and Proposition~\ref{strong iness} lead to the main result of the paper.
\begin{thm}\label{m}
Suppose that a group $\Gamma$ has a finite index subgroup $\Gamma'$ with the 1-Step Stabilization property for $KO_*$  for all $m$ with $k\in K=\{m,m+1\}$ and 
suppose that $\Gamma'$ satisfies the Strong Novikov Conjecture.
Then the original Gromov conjecture holds for all spin $n$-manifolds, $n> 4$, with the fundamental group $\Gamma$.
\end{thm}
\begin{proof}
Let $M$ be a positive scalar curvature spin $n$-manifold with $\pi_1(M)=\Gamma$. Let $M'\to M$ be a finite covering manifold with the fundamental group $\Gamma'$.
The Riemannian metric on $M'$ lifted from $M$ has a positive scalar curvature.
By Rosenberg's theorem (Theorem~\ref{Ros}) $M'$ is $KO$-inessential. 
By Proposition~\ref{KO-iness} $ M'$ is inessential. By Proposition~\ref{strong iness} $ M'$ is strongly inessential. This completes the proof since $M$ and $M'$ have same universal cover.
\end{proof}

\

\

{\em Proof of Proposition~\ref{KO-iness}}

Let $\dim M=n$. By Proposition~\ref{quotients} $B\Gamma/B\Gamma^{(n-1)}$ has the
1-Step Stabilization property property for $KO_*$ with $m=n$ and $k=n$.
We may assume that $M$ has one $n$-dimensional cell and $u(M,M^{(n-1)})\subset (B\Gamma^{(n)},B\Gamma^{(n-1)})$.
Since $u_*([M]_{KO})=0$ in $KO_n (B\Gamma)$, we obtain $q_*u_*([M]_{KO})=0$ in $KO_n (B\Gamma/B\Gamma^{(n-1)})$
where $q:B\Gamma\to B\Gamma/B\Gamma^{(n-1)}$ is the quotient map. 
Let $q'$ denote the restriction  of $q$ to $B\Gamma^{(n+1)}$.
Since $q'_*u_*([M]_{KO})$ lies in the image of the induced homomorphism for
the inclusion $$B\Gamma^{(n)}/B\Gamma^{(n-1)}\to B\Gamma^{(n+1)}/B\Gamma^{(n-1)},$$
by the 1-step Stabilization property applied for $m=n$ with $k=n$ we obtain $q'_*u_*([M]_{KO})=0$  in
$KO_n(B\Gamma^{(n+1)}/B\Gamma^{(n-1)})$.

  Let $\bar u:M/M^{(n-1)}=S^n\to B\Gamma^{(n+1)}/B\Gamma^{(n-1)}$ be
the induced map by the classifying map $u$.  In the commutative diagram
$$
\begin{CD}
ko_n(M) @>j'_*>> ko_n(S^n)@>\bar u_*>> ko_n(B\Gamma^{(n+1)}/B\Gamma^{(n-1)})\\
@. @VperV{\cong}V @V{\cong}VperV\\
KO_n(M) @>j'_*>> KO_n(S^n)@>\bar u_*>> KO_n(B\Gamma^{(n+1)}/B\Gamma^{(n-1)})\\
\end{CD}
$$
the $per$  homomorphisms are isomorphisms in view of Proposition~\ref{moore} and Proposition~\ref{moor}. Since $per\circ j'_*([M]_{ko})=j'_*([M]_{KO})$, this
implies that  $\bar u_*j_*'([M]_{ko})=0$ in $ko_n(B\Gamma^{(n+1)}/B\Gamma^{(n-1)})$. 
In view of the natural transformation of homology theories $ko_*\to H_*(\ ;\mathbb Z)$  it follows that $\bar u_*j_*'([M])=0$ in $H_n(B\Gamma^{(n+1)}/B\Gamma^{(n-1)})$.
Since the homomorphism $q'_*:H_n(B\Gamma^{(n+1)})\to H_n(B\Gamma^{(n+1)}/B\Gamma^{(n-1)})$ is injective, we obtain that $u_*([M])=0$.
Theorem~\ref{ref} completes the proof.
\qed

\

{\em Proof of Proposition~\ref{strong iness}}

Let $\dim M=n$. By Proposition~\ref{quotients} $B\Gamma$ has the 1-Step Stabilization property for $KO_*$ with $k=m+1$.
We may assume that $M$ has a CW complex structure with one $n$-dimensional cell. Let $\psi:D^n\to M$ be its characteristic map.
By Proposition~\ref{ref2} we may assume that the classifying map $u$ satisfies the condition $u(M^{(n-1)})\subset B\Gamma^{(n-2)}$ and $u(M)\subset B\Gamma^{(n-1)}$.

Note that in the diagram
$$
\begin{CD}
ko_n(M/M^{(n-1)}) @>\bar u_*>>ko_n(B\Gamma^{(n-1)}/B\Gamma^{(n-2)}) @>(\bar\phi_{n-1,n})_*>> ko_n(B\Gamma^{(n)}/B\Gamma^{(n-2)})\\
@V\cong VV  @V\cong VV @V\cong VV\\
KO_n(M/M^{(n-1)}) @>\bar u_*>> KO_n(B\Gamma^{(n-1)}/B\Gamma^{(n-2)})@>(\bar\phi_{n-1,n})_*>> KO_n(B\Gamma^{(n)}/B\Gamma^{(n-2)}) \\
\end{CD}
$$
the right vertical arrow is an isomorphism in view of Propositions~\ref{moore} and~\ref{moor}. 
The 1-Step Stabilization property of $B\Gamma/B\Gamma^{(n-2)}$ applied with $m=n-1$ and $k=n$ implies that $(\bar\phi_{n-1,n})_*\bar u_*(1)=0$ for $KO_*$. The above diagram implies that $i_*\bar u_*(1)=0$
for $ko_*$. Lemma~\ref{2-obstruction} completes the proof.
\qed

\

\section{Applications to the  Gromov Conjecture}
For applications of Theorem~\ref{m} to Gromov's conjecture we use the following stronger condition:

$(^*)_s$\ \ \ {\em The  classifying space $B\Gamma$ has the stable homotopy type of  wedge of spheres.}

We recall that this means that the spectrum generated by $B\Gamma$ is the wedge sum of sphere spectra.

We note that in the case when $B\Gamma$ is a finite complex there is $\ell$ such that $\Sigma^\ell B\Gamma$ is homotopy equivalent to $\Sigma^\ell(\vee_j S^{n_j})$~\cite{BCM}.

\begin{prop}\label{imp}
For geometrically finite groups property  $(^*)_s$ implies the 1-Step Stabilization property for any homology theory $h_*$.
\end{prop}
\begin{proof}
Let $f:\Sigma^\ell B\Gamma\to\Sigma^\ell(\vee S^{n_j})$ be a homotopy equivalence with a homotopy inverse map $g$. We may assume that $f$ and $g$ are cellular maps
with respect to the natural CW structure on $\Sigma^\ell S^{n_j}=S^{n_j+\ell}$. Then the maps $\Sigma^\ell\phi_{m,m+1}$ and $\Sigma^\ell\phi_{m,m+1}(g|)(f|)$ 
are homotopic where $f|$ and $g|$ are the
restrictions of $f$ and $g$ to the $(m+\ell)$-skeleton. Then
for any nontrivial element $a\in h_{k+\ell}(\Sigma^\ell B\Gamma^{(m)})$ with $(\Sigma^\ell\phi_{m,m+1})_*(a)\ne 0$ we have $(f|)_*(a)\ne 0$. In this diagram
$$
\begin{CD}
h_{k+\ell}(\Sigma^\ell B\Gamma^{(m)}) @>(\Sigma^\ell\phi_{m,m+1})_*>> h_{k+\ell}(\Sigma^\ell B\Gamma^{(m+1)}) @>(\Sigma^\ell\phi_{m+1,\infty})_*>> h_{k+\ell}(\Sigma^\ell B\Gamma)\\
@Vf|_*VV @. @Ag_*A\cong A\\
h_{k+\ell}(\bigvee_{n_j\le m} S^{n_j+\ell}) @>\xi^0_*>> h_{k+\ell}(\bigvee_{n_j\le m+1} S^{n_j+\ell}) @>\xi^1_*>> h_{k+\ell}(\bigvee S^{n_j+\ell}).\\
\end{CD}
$$
Here the maps $\xi^0$ and $\xi^1$ are inclusions of a wedge in a large wedge.
Since $\xi^1_*\xi^0_*$ is injective, it follows that $$(\phi_{m+1,\infty})_*(\phi_{m,m+1})_*(a)\ne 0.$$
Thus, the 1-Step Stabilization property is verified.
\end{proof}

\begin{prop}\label{oper}
Let $\Gamma_1$ and $\Gamma_2$ be geometrically finite groups satisfying the condition $(^*)_s$, then the free product $\Gamma_1\ast\Gamma_2$ and the product $\Gamma_1\times\Gamma_2$ satisfy $(^*)_s$.
\end{prop}
\begin{proof}
Let  $\Sigma^{\ell_1} B\Gamma_1$ and $\Sigma^{\ell_2} B\Gamma_2$ be homotopy equivalent to $\vee S^{n_j}$ and $\vee S^{m_i}$ respectively. We have $ B \left (\Gamma_1\ast \Gamma_2 \right)=B\Gamma_1\vee B\Gamma_2$, and thus for any $k$, $$ \Sigma^k B \left (\Gamma_1\ast \Gamma_2 \right)=\Sigma^kB\Gamma_1\vee \Sigma^k B\Gamma_2 ,$$ Taking $k=max \{\ell_1,\ell_2 \} ,$ we have the result for $\Gamma_1\ast\Gamma_2$.\\\

We note that $ B \left (\Gamma_1\times \Gamma_2 \right)=B\Gamma_1\times B\Gamma_2  $. The suspension of product of CW complexes formula~\cite{Ha}, $\Sigma(X\times Y)=\Sigma X\vee \Sigma Y\vee \Sigma (X\wedge Y)$ gives us 
$$\Sigma^{\ell_1+\ell_2}B \left (\Gamma_1\times \Gamma_2 \right)=\Sigma^{\ell_1+\ell_2}B\Gamma_1\vee \Sigma^{\ell_1+\ell_2} B\Gamma_2 \vee \Sigma^{\ell_1+\ell_2}(B\Gamma_1\wedge B\Gamma_2). $$ 
The first two components on the right side are wedge sums of spheres. Since taking the reduced $k$-suspension is  equivalent to the smash product with $S^k$, we obtain for the remaining component  $$\Sigma^{\ell_1+\ell_2}(B\Gamma_1\wedge B\Gamma_2)=S^{\ell_1+\ell_2}\wedge(B\Gamma_1\wedge B\Gamma_2)=(S^{\ell_1}\wedge B\Gamma_1)\wedge (S^{\ell_2}\wedge B\Gamma_2).$$
Here we used associativity and commutativity of smash product and the fact that $S^{\ell_1+\ell_2}=S^{\ell_1}\wedge S^{\ell_2}$. Thus, the right hand side is the same as $(\Sigma^{\ell_1} B\Gamma_1)\wedge (\Sigma^{\ell_2} B\Gamma_2)$, which is equivalent to $$(\bigvee S^{n_j}) \wedge (\bigvee S^{m_i})=\bigvee (S^{n_j}\wedge (\bigvee S^{m_i}))=\bigvee (S^{n_j}\wedge S^{m_i})=\bigvee S^{n_j+m_i}.$$ 
\end{proof}

\begin{prop}
The surface groups $\pi_1(M_g)$ for all $g\ge0$, satisfy the condition $(^*)_s$.
\end{prop}
\begin{proof}
The surface $M_g$ has one $0$-cell, $2g$ many $1$-cell, one $2$-cell and the cell structure is obtained by attaching the $2$-cell along $[a_1,b_1]\cdot[a_2,b_2]\cdots[a_g,b_g]$ to the $1$-skeleton, which is wedge of $2g$ circles.  Let $\varphi : S^1\to \bigvee_{2g}S^1_i$ be the attaching map for the 2-cell. We claim that $\Sigma \varphi :S^2\to \bigvee_{2g}S^2_i$ is nullhomotopic. Since $H_2(\bigvee_{2g}S^2_i)=\oplus_{i=1}^{2g} H_2(S^2_i)$, by Hurewicz Theorem $\pi_2(\bigvee_{2g} S^2_i)=\oplus_{i=1}^{2g} \pi_2(S^2_i)$. Thus,
 $$[\Sigma \varphi]=([\Sigma(p_1\varphi)],\dots,[\Sigma(p_{2g}\varphi])\in\bigoplus_{i=1}^{2g} \pi_2(S^2_i) $$ where  
$p_i:\bigvee_{2g} S^1\to S^1_i$ is the projection to $i$-th summand. Note that each $ p_{i}\circ  \varphi :S^1\to S^1$ has degree 0. Hence $deg \Sigma (p_{i}\varphi) = deg (p_{i} \varphi) =0$.  Hence $[\Sigma \varphi]=0$ and $\Sigma \varphi$ is null-homotopic. Therefore, $\Sigma M_g$ is homotopy equivalent to $\left(\bigvee_{2g}S_i^2\right)\vee S^3$.
\end{proof}

Let $X$ be a finite simplicial graph with the set of vertices $S=\{s_1,s_2,...,s_n\}$, and the set of edges $E$. We recall, that the right-angled Artin group defined by $X$ is the group with following presentation
$$
A_X=\langle S\mid [s_i,s_j]=1\ \ \text{whenever}\ \ \{s_i,s_j\}\in E\rangle.
$$

The Salvetti complex, $\mathcal{S}_X$, which is a classifying space, $\mathcal{S}_X:=BA_X$ for the group $A_X$ is constructed as follows: Begin with a wedge of circles attached to a point
$x_0$ and labeled by the generators $s_1,\cdots, s_n$. For each edge, $\{s_i,s_j\}\in E$, consider the
$2$-torus obtained by attaching $2$-cell to the wedge of circles corresponding to $s_i$ and $s_j$ along the relation $[s_i, s_j]$. For each triangle in $X$ connecting three vertices $s_i, s_j , s_k$, consider the $3$-torus defined by circles corresponding to them. Continue this process, having a $k$-torus for each set of $k$ mutually commuting generators (i.e., generators spanning a complete subgraph in $X$). The resulting
space is the Salvetti complex $\mathcal{S}_X$ for $A_X$. 
\\

\begin{prop}
Suppose that $K=e^0\cup(\cup_{i\in J} e^{n_i}_i)$ is a subcomplex of an $n$-dimensional torus $T^n=S^1\times\cdots\times S^1$ with the product CW complex structure and $S^1=e^0\cup e^1$.
Then for sufficiently large $\ell$ the reduced $\ell$-suspension $\Sigma^\ell K$ is homotopy equivalent to the wedge of spheres $\vee_{i\in J}S^{n_i+\ell}$. 

Moreover, there is a cellular homotopy equivalence $f:\Sigma^\ell K\to\vee_{i\in J}S^{n_i+\ell}$ which takes the cell $\Sigma^\ell e^{n_i}_i$ to $S^{n_i+\ell}$ for all $i\in J$.
\end{prop}
\begin{proof}
The proof is based on the formula $$\Sigma(X\times Y)=\Sigma X\vee \Sigma Y\vee \Sigma (X\wedge Y).$$ We refer to the proof of Proposition 6.5 in~\cite{Dr2} for details.
\end{proof}
\begin{cor}\label{raag}
Any subcomplex of $n$-dimensional torus $T^n$ satisfies $(^*)_s$.
\end{cor}

\begin{prop}
RAAGs satisfy $(^*)_s$.
\end{prop}
\begin{proof}
Let $A_X$ be right-angled Artin group generated by $n$ generators where $X$ is the corresponding simple graph with $n$ vertices $\{s_1,\cdots,s_n\}$. Since the Salvetti complex $\mathcal{S}_X$ is built as a union of $k$-tori, $T^k\subset T^n$ for each complete $k$-subgraph of $X$, clearly $\mathcal{S}_X$ is a subcomplex of $T^n$. Thus Lemma~\ref{raag}, completes the proof.
\end{proof}

In view of  Proposition~\ref{imp}, the condition $(^*)_s$ implies the 1-Step Stabilization property for $KO_*$. Thus, Theorem~\ref{m} implies the following
\begin{cor}
Suppose that a group $\Gamma$ satisfies $(^*)_s$ and the Strong Novikov conjecture holds true for $\Gamma$. Then the Strong Gromov conjecture holds for spin manifolds with the fundamental group $\Gamma$. 
\end{cor}
\begin{cor}
The Strong Gromov Conjecture holds true for spin manifolds with the fundamental group $\Gamma$ where
\begin{itemize}
\item $\Gamma$ is a RAAG,
\item $\Gamma$ is a finite product of orientable surface groups.
\end{itemize}
\end{cor}
\begin{proof}
Since these groups are CAT(0) groups, the Strong Novikov Conjecture holds for them and the result follows.
\end{proof}

In view of Proposition~\ref{m}, if the Strong Gromov Conjecture holds for a group $\Gamma'$ contained in $\Gamma$ as a finite index subgroup, then the original Gromov's 
conjecture holds for $\Gamma$.
Since the fundamental group of non-orientable surface contains as index 2, the fundamental group of an orientable surface, we obtain that
the original Gromov Conjecture holds true for products of non-orientable surface groups.
  \begin{remark}
  It is well known that each RAAG is contained in some right-angled Coxeter groups as finite index subgroup,~\cite{DJ}, hence the Gromov Conjecture holds true for spin manifolds having fundamental groups those right-angled Coxeter groups. 
\end{remark}
\begin{question}\label{Q2}
Does the property $(^*)_s$ implies the  (Strong) Novikov conjecture for a group $\Gamma$?
\end{question}

\subsection{More groups satisfying 1-Step Stabilization Property}

Given an abelian group $G$ and $n\ge 1$, a CW complex $X$ having trivial reduced homology $\tilde H_i(X)=0$ for $i\ne n$ and $H_n(X)=G$ is called a Moore space and is denoted as $M(G,n)$. In the case of a finitely generated group $G$ and $n\ge 2$  the Moore space can be chosen to be the finite wedge of $n$-sphere $S^n$ and $n$-spheres with one $(n+1)$-dimensional cell attached $S^n\cup_\phi
D^{n+1}$.

The condition $(^*)_s$ can be weakened to the following

$(^*)_m$\ \ \ {\em The  classifying space $B\Gamma$ has the stable homotopy type of a wedge of Moore spaces.}

\begin{prop}
For geometrically finite groups property  $(^*)_m$ implies the 1-Step Stabilization property for  any homology theory $h_*$.
\end{prop}
\begin{proof}
Let $f:\Sigma^\ell B\Gamma\to\Sigma^\ell(\vee M(G_j,n_j))$ be a homotopy equivalence with a homotopy inverse map $g$. 
Since $B\Gamma$ is geometrically finite, the groups $G_j$ are finitely generated. We may assume that each $G_j$ is isomorphic to $\Z$ or to $\Z_{m_j}$ with $m_j=p^{k_j}$ for 
a prime number $p$. Thus, $M(G_j,n_j)$ is either an $n_j$-sphere $S^{n_j}$ or an $n_j$-sphere  with $(n_j+1)$-cell attached, $S^{n_j}\cup_{\psi_j}D^{n_j+1}$.
We may assume that $f$ and $g$ are cellular map
with respect to the natural CW structure on $\Sigma^\ell M(G_j,n_j) =M(G_j,n_j+\ell)$. 
Let $m'=m+\ell$, $k'=k+\ell$, and $n_j'=n_j+\ell$. Let $f^m$ and $g^m$ denote the
restrictions of $f$ and $g$ to the $(m+\ell)$-skeletons. Then
for any nontrivial element $a\in h_{k+\ell}(\Sigma^\ell B\Gamma^{(m)})$ with $(\Sigma^\ell\phi_{m,m+1})_*(a)\ne 0$ we have $(f|)_*(a)\ne 0$:

Let $\phi^0=\Sigma^\ell\phi_{m,m+1}$ and $\phi^1=\Sigma^\ell\phi_{m+1,\infty}$.
Consider the diagram
$$
\begin{CD}
h_{k'}(\Sigma^\ell B\Gamma^{(m)}) @>\phi^0_*>> h_{k'}(\Sigma^\ell B\Gamma^{(m+1)}) @>\phi^1_*>> h_{k'}(\Sigma^\ell B\Gamma)\\
@Ag^m_*AA @Vf^{m+1}_*VV @Ag_*A\cong A\\
h_{k'}(\bigvee M(G_j,n_j')^{(m')}) @>\xi^0_*>> h_{k'}(\bigvee M(G_j,n_j')^{(m'+1)}) @>\xi^1_*>> h_{k'}(\bigvee M(G_j,n_j')).\\
\end{CD}
$$
where $\xi^0$ and $\xi^1$ are the inclusions of wedge sums. 
Let $a\in h_{k'}(\Sigma^\ell B\Gamma^{(m)})$ with $\phi^0_*(a)\ne 0$. We need to show that $\phi^1_*\phi^0_*(a)\ne 0$.

Since the maps $\phi^0$ and $\phi^0g^mf^m$ 
are homotopic, we obtain $$g^{m+1}_*\xi^0(f^m_*(a))=\phi^0_*g^m_*(f^m_*(a))=\phi^0_*(a)\ne 0.$$
Therefore, $\xi^0_*(f^m_*(a))\ne 0$. We note that $$\bigvee M(G_j,n_j')^{(m'+1)}=\bigvee_{n_j\le m'}M(G_j,n_j')\vee(\bigvee S^{m'+1})$$
and $\xi^0_*(f^m_*(a))$ lives in the summand $\bigvee_{n_j\le m'}M(G_j,n_j')$ which is a retract of $\bigvee M(G_j,n_j')$.
Hence, $\xi^1_*(\xi^0(f^m_*(a))\ne 0$. Then
$$
\phi^1_*\phi^0_*(a)=\phi^1_*g^{m+1}_*\xi^0_*f^m_*(a)=g_*\xi^1_*\xi^0_*f^m_*(a)\ne 0.
$$

Thus, the 1-Step Stabilization property is verified.
\end{proof}

\begin{prop}
All finitely generated groups with $\dim B\Gamma\le 2$ have property $(^*)_m$ with a finite wedge of Moore spaces. 
\end{prop}
\begin{proof}
Apply Proposition~\ref{moore}.
\end{proof}
Suppose that a geometrically finite group $\Gamma$  has property $(^*)_m$. Then for some $\ell$ the $\ell$-suspension $\Sigma^\ell B\Gamma$ is homotopy equivalent to
a wedge $\vee_{j=1}^s M(G_j, n_j)$ where $G_j=\Z$ or $\Z_{p^{k_j}}$ where $p$ is a prime number and $n_j>\ell$. We call such $\Gamma$ {\em 2-avoiding} if $p^{k_j}\ne 2$ for all $j$.

\begin{prop}\label{msmash}(~\cite{JN} Corollary 6.6)
If  $gcd(k,l)=d$ is odd or $4$ divides $k$, then $M(\Z_k,m)\wedge M(\Z_l,n)$ is homotopy equivalent to $M(\Z_d,m+n)\vee M(\Z_d,m+n+1)$.
\end{prop}

\begin{prop}
The product of geometrically finite groups $\Gamma_1\times\Gamma_2$ has property $(^*)_m$, when each of  $\Gamma_1$, $\Gamma_2$ has property $(^*)_m$, provided one of the groups is 2-avoiding.
\end{prop}
\begin{proof}
Proceeding exactly similar to the proof of Proposition~\ref{oper}, this time $(\Sigma^{\ell_1} B\Gamma_1)\wedge (\Sigma^{\ell_2} B\Gamma_2)$ is homotopy equivalent to $$\left( \vee_{j=1}^s M(G_j, n_j)\right)\bigwedge \left (\vee_{i=1}^t M(G_i, m_i)\right)=\bigvee_{i,j} M(G_j, n_j)\wedge M(G_i, m_i).$$

Since $M(\Z, n)=S^n$, if any of $G_j$ or $G_i$ is $\Z$ we have $M(G_j, n_j)\wedge M(G_i, m_i)$ to be a suspension of one of the Moore spaces, which is again Moore space. It is enough to consider when we have same prime $p$ for both groups $G_j$ and $G_i$, since for $k=rs$ with $r,s$ relatively prime we have $M(\Z_k,n)=M(\Z_r,n)\vee M(\Z_s,n)$. If  $G_j=\Z_{p^j}$ and $G_i=\Z_{p^i}$, where $p$ is some prime, then assuming $i\le j$, we have $gcd(p^j,p^i)=p^i$, thus by Proposition~\ref{msmash} we have, $$M(\Z_{p^j}, n_j)\wedge M(\Z_{p^i}, m_i)= M(\Z_{p^i}, n_j+m_i)\vee M(\Z_{p^i}, n_j+m_i+1).$$ 
Note that if $p=2$,  $1 \le i$ and $2\ \le j$ then Proposition~\ref{msmash} gives the above splitting.
\end{proof}

\begin{cor}\label{c}
The Strong Gromov Conjecture holds for spin manifolds whose fundamental groups are finite products $\Gamma=\Gamma_1\times\dots\times\Gamma_k$
of 2-dimensional finitely generated 2-avoiding groups such that each $\Gamma_i$ satisfies the Strong Novikov Conjecture.
\end{cor}
We note that 2-avoiding torsion free one-relator groups are among examples of above $\Gamma_i$. Since they have finite asymptotic dimension~\cite{BD}, \cite{Dr3}, \cite{T},
they satisfy the coarse Baum-Connes conjecture for both $KU$ and $KO$. Therefore, they satisfy the Strong Gromov Conjecture~\cite{HR}. By the Lyndon-Cockeroft theorem~\cite{L},\cite{C} they are 2-dimensional.

We recall that a 2-dimensional group is 2-avoiding if its abelianization does not contain $\mathbb Z_2$ as a direct summand. 
An affirmative answer to the following question would prove the original Gromov's conjecture
for the products of all 2-dimensional groups.
\begin{question}
Does every 2-dimensional group contain a 2-avoiding finite index subgroup?
\end{question}
We believe that the other required condition for 2-dimensional groups in Corollary~\ref{c},  the Strong Novikov Conjecture, can be derived from results of Mathai~\cite{Mat} and Hanke-Schick~\cite{HS}.

\subsection{Bolotov's question.}
Dmitri Bolotov constructed an  example $M_b$ of  a closed  inessential  spin 4-manifold which is not strongly inessential.
At the end of his paper he asked if any product with a torus $M_b\times T^p$ can carry a metric of positive scalar curvature~\cite{Bo2}.
Here we give a negative answer.
\begin{prop}
The manifolds  $M_b\times T^p$ for all $p> 0$ do not admit a metric with positive scalar curvature.
\end{prop}
\begin{proof}
The fundamental group of Bolotov's manifold is $\Gamma=\mathbb Z\ast\mathbb Z^3$.
By Proposition~\ref{oper} the group $\Gamma\times\Z^p$ satisfies the condition $(^*)_s$. 
The Strong Novikov conjecture holds for the group $\Gamma\times\Z^p$ since it is constructed out of  integers by taking operation of the product and the free product. 
Note that $M_b\times T^p$ is spin and inessential. By Theorem~\ref{m} the Strong Gromov conjecture holds for $M_b\times T^p$ for $p>0$.
Therefore, $M_b\times T^p$ cannot  carry a metric of positive scalar curvature. Hence $M_b$
cannot carry a metric of positive scalar curvature as well.
\end{proof}


\begin{thebibliography}{CJY}

\bibitem[Ba]{Ba}
I. Babenko, {\em Asymptotic invariants of smooth manifolds.} Russian Acad. Sci.
Izv. Math. 41 (1993), 1 -38.

\bibitem[BD]{BD} G.Bell, A. Dranishnikov, {\em On asymptotic dimension of  groups acting on trees}, Geom. Dedicata, 103 (2004), 89-101.

\bibitem[Bo1]{Bo1}
D. Bolotov, {\em Macroscopic dimension of 3-manifolds}, Mathematical Physics, Analysis and Geometry, vol 6, issue 3 (2003), 291-299.


\bibitem[Bo2]{Bo2}
D. Bolotov, {\em Gromov's macroscopic dimension conjecture}, Algebraic \& Geometric Topology, 6 (2006), 1669-1676.

\bibitem[Bo3]{Bo3}
D. Bolotov, {\em About the macroscopic dimension of certain PSC-manifolds},  Algebraic \& Geometric Topology, 9 (2009), 21-27.

\bibitem[BD1]{BD1}
    D. Bolotov, A. Dranishnikov
\newblock {\em   On Gromov's scalar curvature conjecture, }
\newblock {Proc. AMS
 138 (2010), NO.4, 1517 - 1524.}

\bibitem[BD2]{BD2}
    D. Bolotov, A. Dranishnikov
\newblock {\em   On Gromov's conjecture for totally non-spin manifolds, }
\newblock {Journal of Topology and Analysis Vol.8, No.4 (2016) 571-587.}

\bibitem[Bou]{Bou} {\em The localization of spectra with respect to homology}, Topology, 18 (1979) 257-281.

\bibitem[BCM]{BCM} R.R Bruner, F. R. Cohen, C. A. McGibbon, {\em On stable homotopy equivalences}, Quart. J. Math.Oxford Ser. (2) 46 (1995), no 181, 11-20.


\bibitem[Br]{Br} K. Brown,  Cohomology of groups, Springer 1982.

\bibitem[C]{C} W. H. Cockeroft, {\em On two-dimensional aspherical complexes}, PAMS (3) 4 (1954), 375-384.

\bibitem[DD]{DD}
M. Daher, A. Dranishnikov, {\em  On Macroscopic dimension of non-spin  4-manifolds}, Journal of Topology and Analysis, 

\bibitem[DJ]{DJ} M. W. Davis, T. Januszkiewicz, \emph{Right-angled Artin groups are commensurable with right-angled Coxeter groups}. 
Journal of Pure and Applied Algebra 153 (2000) 229–235.


\bibitem[Dr1]{Dr1} A. Dranishnikov,
{\em On Gromov's positive scalar curvature conjecture for virtual duality groups}, Journal of Topology and Analysis, Vol. 6, No. 3 (2014) 397-419.

\bibitem[Dr2]{Dr2} A. Dranishnikov, \emph{Positive scalar curvature, macroscopic dimension, and inessential manifolds }, in "Perspectives in Scalar Curvature" vol. 2  Ch 7. (2023), 231-248.

\bibitem[Dr3]{Dr3} A. Dranishnikov, {\em On asymptotic dimension of amalgamated products and right-angled Coxeter groups}, Algebr.Geom.Topol. 8 (2008), no.3, 1281-1293.


\bibitem[G1]{G1} M. Gromov  {\em Filling Riemannian manifolds}. J. Differential Geom. 18 (1983), no. 1, 1-147.

\bibitem[G2]{G2}
 M. Gromov
 \newblock {Positive curvature, macroscopic dimension, spectral gaps and higher signatures},
\newblock {Functional analysis on the eve of the 21st century. Vol. II, Birkhauser, Boston,
MA, (1996).}

\bibitem[G3]{G3}
M. Gromov, Metric structures for Riemannian and non-Riemannian spaces, Progress
in Math. 152, Birkhauser, Boston (1999).


 \bibitem[GL]{GL}
 M. Gromov, H.B. Lawson,
 \newblock {\em Positive scalar curvature and the Dirac operator on
 complete Riemannian manifolds},
 \newblock {Publ. Math. I.H.E.S. {\bf 58}  (1983), 295-408.}

\bibitem[Ha]{Ha}
 A. Hatcher, Algebraic Topology, Cambridge University Press, Cambridge 2002.

\bibitem[HS]{HS} B. Hanke, Th. Schick, {\em The strong Novikov conjecture for low degree
cohomology}, Geom. Dedicata (2008) 135, 119-127.

\bibitem[HR]{HR} N. Higson, J. Roe, On the coarse Baum-Connes conjecture, London Math. Soc. Lecture Note Ser., 227
Cambridge University Press, Cambridge, 1995.

\bibitem[JN]{JN} J. Neisendorfer, Primary Homotopy Theory, Memoirs AMS 232, 1980.

\bibitem[L]{L} R. C. Lyndon, {\em Cohomology theory of groups with a single defining relation}, Ann. of Math. (2) 52 (1950), 650-665.

\bibitem[M]{M} M. Marcinkowski, {\em Gromov positive scalar curvature conjecture and rationally inessential macroscopically large manifolds
M Marcinkowski}, Journal of Topology 9 (1), 105-116.


\bibitem[Mat]{Mat} V. Mathai, {\em The Novikov Conjecture for Low Degree Cohomology Classes}, Geometriae Dedicata 99, (2003), 1-15.


\bibitem[R1]{R1}
J. Rosenberg,
{\em C*-algebras, positive scalar curvature, and the Novikov conjecture.}
Publications Mathématiques de l'IHÉS, 58 (1983), p. 197-212.

\bibitem[R2]{R2}
J. Rosenberg, {\em Manifolds of positive scalar curvature: a progress report}, Surv.Differ. Geom., 11, International Press, Somerville, MA, 2007, 259-294.

\bibitem[Ru]{Ru}
Yu. Rudyak,  On Thom spectra, orientability and cobordism. Springer, 1998.


\bibitem[SY]{SY} R. Schoen and S. T. Yau, {\em Positive Scalar Curvature and Minimal Hypersurface Singularities} Preprint (2017) arXiv:1704.05490.

\bibitem[TtD]{TtD} Tammo tom Dieck, Algebraic Topology, EMS 2008.

\bibitem[T]{T} P. Tselekidis, {\em Asymptotic Dimension of Graphs of Groups and One Relator Groups.} Preprint (2020) arXiv:1905.07925.



\bibitem[Yu]{Yu} Guoliang Yu, {\em The Novikov conjecture}, Russian Math. Surveys 74 (2019), no 3, 525-541.

\end{thebibliography}
\end{document}